\documentclass[a4paper]{amsart}
\usepackage[T2A]{fontenc}
\usepackage{comment}
\usepackage{hyperref}
\usepackage{amsmath, amsthm}
\usepackage{mathtools}
\usepackage{amssymb, mathabx}
\usepackage{color}
\usepackage{tikz-cd}
\usepackage{longtable}

\newtheorem{theorem}{Theorem}[section]
\newtheorem{proposition}[theorem]{Proposition}
\newtheorem{lemma}[theorem]{Lemma}
\newtheorem{corollary}[theorem]{Corollary}

\theoremstyle{definition}
\newtheorem{definition}[theorem]{Definition}
\newtheorem{example}[theorem]{Example}
\newtheorem{construction}[theorem]{Construction}

\theoremstyle{remark}

\numberwithin{equation}{section}

\def \K {\mathcal{K}}
\def\sK{\mathcal K}

\def \Ker {\operatorname{Ker}}

\def \Z {\mathbb{Z}}

\def \C {\mathbb{C}}
\def \R {\mathbb{R}}
\def \gr {\operatorname{gr}}

\newcommand{\mb}[1]{{\textbf {\textit#1}}}

\def\Ab{\mathit{Ab}}
\def\cir{\mbox{\it Cir\/}}
\def\FL{\mathit{FL}}
\def\gr{\mathop{\mathrm{gr}}\nolimits}
\def\raag{\mbox{\it RA\/}}
\def\racg{\mbox{\it RC\/}}
\def\RCK{{\operatorname{\mathit{RC}}_\mathcal{K}}}

\def\ass{\text{\sc ass}}
\def\cat{\text{\sc c}}
\def\catK{\mathop{\text{\sc cat}}(\mathcal K)}

\def\grp{\text{\sc grp}}

\def\hpf{\text{\sc hpf}}
\def\lie{\text{\sc lie}}

\def\tgp{\text{\sc tgp}}
\def\tmn{\text{\sc tmn}}
\def\top{\text{\sc top}}

\def\Hom{\mathop\mathrm{Hom}\nolimits}

\def\colim{\mathop\mathrm{colim}\nolimits}
\def\lim{\mathop\mathrm{lim}\nolimits}
\def\hocolim{\mathop\mathrm{hocolim}\nolimits}

\def\pt{\mathit{pt}}

\def\rk{\mathcal R_\mathcal K}
\def\zk{\mathcal Z_\mathcal K}

\title{Polyhedral products, graph products and $p$-central series}
\author{Taras Panov}
\author{Temurbek Rahmatullaev}

\address{Department of Mathematics and Mechanics, Moscow
State University, Russia;\newline
Steklov Mathematical Institute of Russian Academy of Sciences, Moscow, Russia;\newline
Institute for Theoretical and Experimental Physics of National Research Centre ``Kurchatov Institute'', Moscow, Russia;\newline
National Research University Higher School of Economics, Moscow, Russia}
\email{tpanov@mech.math.msu.su}
\email{raxtemur@gmail.com}



\thanks{The work of Taras Panov (Sections~2, 3 and~7) was supported by the Russian Science Foundation under grant no.~23-11-00143, http://rscf.ru/en/project/23-11-00143/, and performed at Steklov Mathematical Institute of Russian Academy of Sciences. The work of Temurbek Rakhmatullaev (Sections 4--6) was as carried out within the project ``Mirror Laboratories'' of HSE University, Russian Federation.}

\dedicatory{To Victor Matveevich Buchstaber
on the occasion of his 80th birthday}

\begin{document}

\begin{abstract}
We relate polyhedral products of topological spaces to graph products of groups. The loop homology algebras of polyhedral products are identified with the universal enveloping algebras of the Lie algebras associated with central series of graph products. By way of application, we describe the restricted Lie algebra associated with the lower $2$-central series of a right-angled Coxeter group and identify its universal enveloping algebra with the loop homology of the Davis--Januszkiewicz space.
\end{abstract}

\maketitle

\section{Introduction} 
The polyhedral product space $(\mb X,\mb A)^\sK$ is defined for a sequence of pairs of topological spaces $(\mb X,\mb A)=((X_1,A_1),\ldots,(X_m,A_m))$ and a simplicial complex $\sK$ on a finite set $[m]=\{1,\ldots,m\}$. When each $A_i$ is a point, the polyhedral product $\mb X^\sK=(\mb X,\pt)^\sK$ interpolates between the $m$-fold wedge $X_1\vee\cdots\vee X_m$ and the $m$-fold product $X_1\times\cdots\times X_m$. The study of polyhedral products takes its origin in toric topology~\cite{bu-pa15} and has recently developed into an active area within homotopy theory~\cite{b-b-c20}. Important examples are the Davis--Januszkiewicz space $(\C P^\infty)^\sK$, the moment-angle complex $\zk=(D^2,S^1)^\sK$ and its real analogue $\mathcal R_\sK=(D^1,S^0)^\sK$. The latter is a cubic complex that features in geometric group theory.

The polyhedral product $\mb X^\sK$ can also be defined for a sequence of objects $\mb X=(X_1,\ldots,X_m)$ in an arbitrary category $\cat$ with finite products and colimits, or in a symmetric monoidal category with finite colimits. For many categories of algebraic objects, like groups $\grp$ or associative algebras $\ass$, the polyhedral product $\mb X^\sK$ depends only on the one-skeleton (graph) $\sK^1=\Gamma$ and is referred to as the \emph{graph product} (see Proposition~\ref{fatiso}). This reflects the fact that a set of pairwise commuting elements of a group (algebra) generates a commutative subgroup (subalgebra). On the contrary, the polyhedral product of topological spaces is not determined by the one-skeleton~$\sK^1$. For example, when $\sK=\partial\Delta_{[m]}$, the boundary of simplex, the polyhedral product $\mb X^{\partial\Delta_{[m]}}$ is known in homotopy theory as the \emph{fat wedge} of $(X_1,\ldots,X_m)$, and it is different from the $m$-fold product $\mb X^{\Delta_{[m]}}=X_1\times\cdots\times X_m$.

Graph products are familiar objects in combinatorial and geometric group theory. Well-known examples are the \emph{right-angled Artin group}
\[
  \raag_\sK=F(a_1,\ldots,a_m)\big/ 
  \bigl(a_ia_j=a_ja_i \text{ for }\{i,j\}\in\sK\bigr),
\]
which is the graph product $\Z^\sK$,
and the \emph{right-angled Coxeter group}
\[
  \racg_\sK=F(a_1,\ldots,a_m)\big/ \bigl(a_i^2=1,\;a_ia_j=a_ja_i
  \text{ for }\{i,j\}\in\sK\bigr),
\]
which is the graph product $\Z_2^\sK$. (We denote the group of integers by $\Z$ and the group/field of residues modulo prime $p$ by~$\Z_p$.)

A number of results in homotopy theory of polyhedral products and geometric group theory can be interpreted as the property of some functor to preserve graph products. These include the constructions of the classifying spaces for right-angled Artin and Coxeter groups, the description of their cohomology, and the description of the loop homology algebras of polyhedral products. We give a more detailed account in Section~\ref{secpg}. Group-theoretic and homotopy-theoretic results of this sort usually come in pairs, with similar formulations but rather different proofs. For example, there are very similar descriptions of the commutator subgroup $\racg'_\sK$ of a right-angled Coxeter group $\racg_\sK$~\cite{pa-ve16} and the commutator subalgebra $H_*(\varOmega\zk)$ of the Pontryagin algebra (loop homology) $H_*(\varOmega(\C P^\infty)^\sK)$~\cite{g-p-t-w16}. Furthermore, when $\sK$ is a flag complex, the commutator subgroup $\racg'_\sK$ and the commutator subalgebra $H_*(\varOmega\zk)$ are free if and only if the one-skeleton $\sK^1$ is a chordal graph. There is also a similarity in the criteria for $\racg'_\sK$ and $H_*(\varOmega\zk)$ to be a one-relator group and algebra, respectively~\cite{g-i-p-s22}. Graph products can also be defined for simplicial groups~\cite{cai}, and their classifying spaces are related to polyhedral products in the simplicial homotopy-theoretical setting.

In this paper we give evidence that group-theoretical results on graph products and homotopy-theoretical results on polyhedral products can be linked to each other via central series of groups and their associated graded Lie algebras. This manifests itself most directly in the case of right-angled Artin groups~$\raag_\sK$. Namely, the Lie algebra associated with the lower central series $\gamma(\raag_\sK)=\{\raag_\sK=\gamma_1(\raag_\sK)\supset\gamma_2(\raag_\sK)
\supset\ldots\}$ is the ``partially commutative'' quotient of a free Lie algebra:
\[
  \gr_\bullet\gamma(\raag_\sK)\cong
  \FL(v_1,\ldots,v_m)\big/
  \bigl([v_i,v_j]=0\;
  \text{for }\{i,j\}\in\sK\bigr).
\]
This extends the classical result of Magnus on the lower central series of a free group and is proved by a similar method based on the Magnus map~\cite{du-kr92, pa-su06, wade15, b-h-s20}. The universal enveloping algebra of $\gr_\bullet\gamma(\raag_\sK)$ is the partially commutative associative algebra. The latter is isomorphic to the loop homology of the polyhedral product of spheres $(S^{2p+1})^\sK$ by the result of~\cite{bu-go11}:
\[
  \mathcal U(\gr_\bullet\gamma(\raag_\sK))\cong
  H_*(\varOmega(S^{2p+1})^\sK;\Z).
\]

The lower central series of a right-angled Coxeter group $\racg_\sK$ is more subtle. The associated Lie algebra $\gr_\bullet\gamma(\racg_\sK)$ is a Lie algebra over~$\Z_2$ that is not finitely presented~\cite{vere22,ve-ra24}. It cannot be readily expressed as a graph product, and there is no explicit connection to loop homology of polyhedral products. \emph{Restricted} Lie algebras associated with \emph{$p$-central} series better serve this purpose.

The \emph{lower $p$-central series} of a group $G$ is the fastest descending central series 
$\{G=
\gamma^{[p]}_1(G)\supset\gamma^{[p]}_2(G)\supset\cdots\}$
with the property that $\bigl(\gamma^{[p]}_k(G)\bigr)^p
\subset \gamma^{[p]}_{kp}(G)$ for $k=1,2,\ldots$. The associated graded abelian group
\[
  \gr_\bullet\gamma^{[p]}(G)
  =\bigoplus_{k\ge1}
  \gamma^{[p]}_k(G)/\gamma^{[p]}_{k+1}(G)
\]
is a \emph{restricted Lie algebra} or \emph{$p$-Lie algebra}: it is a $\Z_p$-module with a Lie bracket and a $p$-power operation $x\mapsto x^{[p]}$. The details are given in Section~\ref{secres}.

In  Theorem~\ref{2lieracg} we give the following explicit presentation for the $2$-Lie algebra associated with the lower $2$-central series of $\racg_\sK$:
\[
  \gr_\bullet\gamma^{[2]}(\racg_\sK)\cong
  \FL_2(u_1,\ldots,u_m)\big/\bigl(u_i^{[2]}=0,\; 
  [u_i,u_j]=0\text{ for }\{i,j\}\in\sK\bigr).
\]
The right hand side is the graph product of trivial $2$-Lie algebras $\Z_2\langle u\rangle=\FL_2(u)/(u^{[2]}=0)$ with one generator. There is also a generalisation to graph products of elementary abelian $p$-groups for an arbitrary prime~$p$ (Theorem~\ref{rlazpk}). The proof uses the result of Quillen~\cite{quil68} relating the lower $p$-central series of a group to the filtration of its group algebra by the powers of the augmentation ideal alongside with the construction of an analogue of the Magnus map. The universal enveloping algebra $\mathcal U_2(\gr_\bullet\gamma^{[2]}(\racg_\sK))$ is identified with the loop homology of $(\C P^\infty)^\sK$ with $\Z_2$-coefficients (Proposition~\ref{ueracg}).
As a consequence, we obtain the following relation between the fundamental group of the polyhedral product of $\R P^\infty$ and the loop homology of the polyhedral product of $\C P^\infty$ for flag complexes~$\sK$:
\[
  \mathcal U_2\bigl(\gr_\bullet\gamma^{[2]}\pi_1(\R P^\infty)^\sK\bigr)\cong
  H_*\bigl(\varOmega(\C P^\infty)^\sK;\Z_2\bigl).
\]
Recall that $\pi_1((\R P^\infty)^\sK)=(\pi_1(\R P^\infty))^\sK=\racg_\sK$.

In the Appendix we give a comparative table of group-theoretical and homotopy-theoretical results on polyhedral products and graph products.

The authors are grateful to the anonymous referee for very helpful comments and suggestions.

\section{Polyhedral and graph products}\label{secpg}
Let $\K$ be a simplicial complex on a finite ordered 
set $[m]=\{1,\dots,m\}$. We
assume that the empty set $\varnothing$ and all one-element subsets $\{i\}\subset[m]$ belong to~$\sK$. We refer to $I=\{i_1,\ldots,i_k\}\in\mathcal K$ as a \emph{simplex} or  \emph{face} of~$\K$.
The \emph{face category}
$\catK$ has simplices $I\in\sK$ as objects and inclusions $I\subset J$ as morphisms.

\begin{construction}[polyhedral product]\label{polpr}
Let
\[
  (\mb X,\mb A)=\{(X_1,A_1),\ldots,(X_m,A_m)\}
\]
be a sequence of $m$ pairs of pointed topological spaces, $\pt\in
A_i\subset X_i$, where $\pt$ denotes the basepoint. For each subset $I\subset[m]$, we set
\[ 
  (\mb X,\mb A)^I=\bigl\{(x_1,\ldots,x_m)\in
  \prod_{k=1}^m X_k\colon\; x_k\in A_k\;\text{for }k\notin I\bigl\}
\]

Let $\top$ denote the category of pointed
topological spaces. Consider the $\catK$-diagram
\[
\begin{aligned}
  \mathcal T_\sK(\mb X,\mb A)\colon \catK&\longrightarrow \top,\\
  I&\longmapsto (\mb X,\mb A)^I,
\end{aligned}
\]
which maps the morphism $I\subset J$ of $\catK$ to the
inclusion of spaces $(\mb X,\mb A)^I\subset(\mb X,\mb A)^J$. 

The \emph{polyhedral product} of $(\mb X,\mb A)$
corresponding to $\sK$ is given by
\[
(\mb X,\mb A)^\sK=
  \mathop{\mathrm{colim}}\nolimits^\top
  \mathcal T_\sK(\mb X,\mb A)=\mathop{\mathrm{colim}}_{I\in\sK}\nolimits^\top
  (\mb X,\mb A)^I.
\]
Here $\colim$ denotes the \emph{colimit functor} from the category of
$\catK$-diagrams of topological spaces to the
category~$\top$. By definition, $\colim$ is the left adjoint to the constant diagram functor.
Explicitly,
\[
  (\mb X,\mb A)^{\sK}=\bigcup_{I\in\mathcal K}(\mb X,\mb A)^I=
  \bigcup_{I\in\mathcal K}
  \Bigl(\prod_{i\in I}X_i\times\prod_{i\notin I}A_i\Bigl)\subset\prod_{i=1}^m X_i.
\]
When each $A_i$ is a point,
we use the abbreviated notation $\mb X^\sK$ for $(\mb X,\pt)^\sK$.

As a functor, the polyhedral product $\mb X^\sK$ can be defined in a symmetric monoidal category~$\cat$ with finite colimits. Consider a sequence of objects $\mb X=(X_1,\ldots,X_m)$ in~$\cat$, set $\mb X^I=\prod_{i\in I}X_i$, where the product is taken with respect to the monoidal structure, and consider the diagram
\[
\begin{aligned}
  \mathcal C_\sK(\mb X)\colon
  \catK&\longrightarrow\cat,\quad&
  I\longmapsto \mb X^I,
\end{aligned}
\]
taking a morphism $I\subset J$ to the canonical morphism $\mb X^I\to\mb X^J$. The latter is defined using the unit object in~$\cat$. Then define the polyhedral product
\begin{equation}\label{mcprod}
  \mb X^{\sK}=\colim^\cat\mathcal C_\sK(\mb
  X)=\colim^\cat_{I\in\sK}\mb X^I.
\end{equation}
\end{construction}

The \emph{full subcomlex} of $\sK$ on a subset $I\subset[m]$ is  $\sK_I=\{J\in\sK\colon J\subset I\}$.

A \emph{missing face} (or a \emph{minimal non-face}) of $\sK$ is a subset $I\subset[m]$ such that $I$ is not a simplex of~$\sK$, but every proper subset of $I$ is a simplex of~$\sK$. Each missing face corresponds to a full subcomplex $\partial\Delta_I\subset\sK$, where $\partial\Delta_I$ denotes the boundary of the simplex on the vertex set~$I$.

A simplicial complex $\sK$ is called a \emph{flag complex} if each of its missing faces consists of two vertices. Equivalently, $\sK$ is flag if any set of vertices of $\sK$ which are pairwise connected by edges spans a face. 

Complete subgraphs in a simple graph $\Gamma$ form a flag simplicial complex $\sK(\Gamma)$, called the \emph{clique complex} of~$\Gamma$.
Every flag complex $\sK$ is the clique complex of its 1-skeleton~$\sK^1$.

The polyhedral product $\mb X^{\partial\Delta_{[m]}}$ of pointed topological spaces $X_1,\ldots,X_m$ is known as their \emph{fat wedge}. When $m=2$, it is the wedge $X_1\vee X_2$.

\begin{definition}
A symmetric monoidal category $\cat$ with finite colimits is said to be \emph{fat} if the canonical morphism 
\[
  \mb X^{\partial\Delta_{[m]}}\to\prod_{i=1}^m X_i
\]
is an isomorphism for any objects $X_1,\ldots,X_m$ and $m\ge3$.
\end{definition}

The following fat categories will be used throughout the paper:
\begin{itemize}
\item[$\grp$:] groups (the monoidal product is cartesian, the coproduct is the free product);

\item[$\tmn$:] topological monoids (the monoidal product is cartesian, the coproduct is the free product);

\item[$\tgp$:] topological groups (a full subcategory of $\tmn$);

\item[$\ass$:] associative algebras with unit over a commutative ring~$R$ (the monoidal product is the tensor product, the coproduct is the free product);

\item[$\lie$:] Lie algebras over~$R$ (the monoidal product is the direct sum, the coproduct is the free product).
\end{itemize}
In each of these categories, a set of pairwise commuting elements of a group (algebra) generates a commutative subgroup (subalgebra). On the other hand, the category $\top$ is clearly not fat.

\begin{proposition}\label{fatiso}
In a fat category $\cat$, the morphism $\mb X^{\sK^1}\to\mb X^\sK$ is an isomorphism for any simplicial complex~$\sK$ and sequence of objects $\mb X=(X_1,\ldots,X_m)$ .
\end{proposition}
\begin{proof}
We add simplices $I$ with $|I|\ge3$ to $\sK^1$ consecutively until we end up at~$\sK$, and use induction on the number of added simplices. The base of induction is $\sK=\sK^1$ and is clear.
Suppose $\sK=\sK'\cup I$, where $|I|\ge3$, and $\mb X^{\sK^1}\to\mb X^{\sK'}$ is an isomorphism by induction assumption. We have a pushout square of simplicial complexes and the corresponding pushout of polyhedral products:
\[
\begin{tikzcd}
  \partial\Delta_I \ar{r}\ar{d} & \Delta_I \ar{d}\\ 
  \sK' \ar{r} & \sK,
\end{tikzcd}
\qquad
\begin{tikzcd}
  \mb X^{\partial\Delta_I} \ar{r}\ar{d} & 
  \mb X^{\Delta_I} \ar{d}\\ 
  \mb X^{\sK'} \ar{r} & \mb X^{\sK}.
\end{tikzcd}
\]
Since $\mb X^{\partial\Delta_I}\to\mb X^{\Delta_I}$ is an isomorphism, $\mb X^{\sK'} \to \mb X^{\sK}$ is also an isomorphism, and so is $\mb X^{\sK^1}\to\mb X^\sK$.
\end{proof}

\begin{definition}
For a sequence $\mb X=(X_1,\ldots,X_m)$ of objects in $\cat$ and a finite simple graph $\Gamma$ on $m$ vertices, the \emph{graph product} $\mb X^\Gamma$ is defined as the polyhedral product $\mb X^{\sK(\Gamma)}$.
\end{definition}

By Proposition~\ref{fatiso}, if $\cat$ is fat, then $\mb X^\Gamma$ is canonically isomorphic to $\mb X^\sK$ for any $\sK$ with $\sK^1=\Gamma$.

\begin{example}[graph product of groups]\label{cgrpr}
Let $\mb
G=(G_1,\ldots,G_m)$ be a sequence of $m$ (topological) groups. 
%
Since the category $\tgp$ is fat, we have 
\[
  \mb G^{\sK^1}=\mb G^{\sK}
  =\colim^\tgp_{I\in\sK}\mb G^I.
\]
Explicitly,
\[
  \mb G^{\sK}=\mathop{\mbox{\Huge$\star$}}_{k=1}^m G_k\big/(g_ig_j=g_jg_i\,\text{ for
}g_i\in G_i,\,g_j\in G_j,\,\{i,j\}\in\sK),
\]
where $\mathop{\mbox{\Huge$\star$}}_{k=1}^m G_k$ denotes the free product of the groups~$G_k$.
There are canonical injective homomorphisms $\mb G^I\to\mb G^\sK$ for $I\in\sK$.
\end{example}

There will be several occasions throughout the paper when a functor between two categories preserves graph products. This is certainly the case when the functor preserves finite products and colimits. Other examples include the classifying space functor and the loop homology functor, considered below.

Given a topological group $G$, consider the universal $G$-bundle $EG\to BG$, where $EG$ is the universal $G$-space and $BG$ the classifying space for~$G$.

Following the notation from Example~\ref{cgrpr}, the
classifying space $B\mb G^I$ is the product of $BG_i$ over $i\in
I$. We therefore have the polyhedral product $(B\mb G)^\sK$
corresponding to the sequence of pairs $(B\mb
G,\pt)=\{(BG_1,\pt),\ldots,(BG_m,\pt)\}$.
Similarly, we have the polyhedral product $(E\mb G,\mb G)^\sK$
corresponding to the sequence of pairs $(E\mb G,\mb
G)=\{(EG_1,G_1),\ldots,(EG_m,G_m)\}$. Here each $G_i$ is included
in $EG_i$ as the fibre of $EG_i\to BG_i$ over the
basepoint.

\begin{proposition}[{\cite[Proposition~5.1]{p-r-v04},
\cite[4.3.10]{bu-pa15}}]\label{ragfib}
There is a homotopy fibration of polyhedral products
\begin{equation}\label{hofib}
  (E\mb G,\mb G)^\sK\longrightarrow (B\mb
  G)^\sK\longrightarrow\prod_{i=1}^m BG_i.
\end{equation}
\end{proposition}

The classifying space functor
\[
  B\colon\tgp\to\top
\]
does not preserve colimits over small categories in general, but it does preserve \emph{homotopy colimits} under appropriate model structures on $\top$ and $\tgp$, see~\cite[Theorem~7.12]{p-r-v04}. The following specification describes the relationship between polyhedral products of spaces and graph products of groups.

\begin{theorem}[{\cite[Proposition~6.1, Theorem~7.17]{p-r-v04}}]
There is a commutative square of natural maps
\[
\begin{tikzcd}
  \hocolim^\top_{I\in\sK} B\mb G^I \ar{r}{\simeq} 
  \ar{d}{\simeq}&
  B\hocolim^\tgp_{I\in\sK} \mb G^I \ar{d}\\
  \colim^\top_{I\in\sK} B\mb G^I \ar{r} &
  B\colim^\tgp_{I\in\sK}\mb G^I,
\end{tikzcd}
\]
in $\top$, where the top and left maps are homotopy equivalences, and the right map is a homotopy equivalence when $\sK$ is a flag complex.
\end{theorem}

\begin{corollary}\label{bcommut}
The natural map 
\[
  (B\mb G)^\sK\to B(\mb G^{\sK})
\]
is a homotopy equivalence when $\sK$ is a flag complex. That is, the classifying space functor preserves graph products up to homotopy equivalence. 
\end{corollary}

\begin{proposition}\label{basph}
If $\sK$ is flag and $\mb G=(G_1,\ldots,G_m)$ consists of discrete groups, then the polyhedral product $(B\mb G)^\sK$ is an aspherical space with $\pi_1((B\mb G)^\sK)\cong\mb G^{\sK}$. 
\end{proposition}

\begin{proof}
Follows from Corollary~\ref{bcommut}. Alternatively, this can be proved using the standard Whitehead technique of gluing aspherical spaces. The argument below was outlined in~\cite[Theorem~10]{ki-ro80} in the case of right-angled Artin groups. 

By Whitehead's theorem~\cite[Theorem~5]{whit39}, if $P=P_1\cup P_2$ is a CW-complex, $P_1$ and $P_2$ are CW-subcomlexes of~$P$ such that (a) $P_1$, $P_2$ and $P_1\cap P_2$ are aspherical and (b) both homomorphisms $\pi_1(P_1\cap P_2)\to\pi_1(P_1)$ and $\pi_1(P_1\cap P_2)\to\pi_1(P_2)$ are injective, then $P$ is also aspherical. We use induction on the number of vertices~$m$ of~$\sK$. The base of induction is $m=1$ and is clear. Consider the subcomplexes 
$\mathop\mathrm{st}_{\sK}\{m\}=\{I\colon I\cup\{m\}\in\sK\}$ (the star of the $m$th vertex of~$\sK$) and $\sK_{[m-1]}$ (the full subcomplex on the first $m-1$ vertices). Their intersection is the 
the link of the $m$th vertex: 
\[
  \mathop\mathrm{st}\nolimits_{\sK}\{m\}\cap \sK_{[m-1]}=
  \mathop\mathrm{lk}\nolimits_{\sK}\{m\}=\{J\colon J\not\ni m,\,J\cup\{m\}\in\sK\}.
\]  
We have a pushout square of simplicial complexes and the corresponding pushout of polyhedral products:
\[
\begin{tikzcd}
  \mathop\mathrm{lk}\nolimits_{\sK}\{m\} \ar{r}\ar{d} & 
  \sK_{[m-1]} \ar{d}\\ 
  \mathop\mathrm{st}\nolimits_{\sK}\{m\} \ar{r} & \sK,
\end{tikzcd}
\qquad
\begin{tikzcd}
  (B\mb G)^{\mathop\mathrm{lk}\nolimits_{\sK}\{m\}} \ar{r}\ar{d} & 
  (B\mb G)^{\sK_{[m-1]}} \ar{d}\\ 
  (B\mb G)^{\mathop\mathrm{st}\nolimits_{\sK}\{m\}} \ar{r} & 
  (B\mb G)^{\sK}.
\end{tikzcd}
\]
We set $P_1=(B\mb G)^{\sK_{[m-1]}}$ and $P_2=(B\mb G)^{\mathop\mathrm{st}\nolimits_{\sK}\{m\}}$ in Whitehead's theorem. Note that $(B\mb G)^{\mathop\mathrm{st}\nolimits_{\sK}\{m\}}=(B\mb G)^{\mathop\mathrm{lk}\nolimits_{\sK}\{m\}}\times B G_m$ by the definition of polyhedral product. Since full subcomplexes and links in a flag complex are flag, the polyhedral products $P_1$,
$P_1\cap P_2=(B\mb G)^{\mathop\mathrm{lk}\nolimits_{\sK}\{m\}}$ and $P_2=(P_1\cap P_2)\times B G_m$ are aspherical by induction assumption. The homomorphism $\pi_1(P_1\cap P_2)\to\pi_1(P_2)$ is clearly injective. Furthermore, $\mathop\mathrm{lk}\nolimits_{\sK}\{m\}$ is a full subcomplex of $\sK_{[m-1]}$, so the polyhedral product $(B\mb G)^{\mathop\mathrm{lk}\nolimits_{\sK}\{m\}}$ is a retract of $(B\mb G)^{\sK_{[m-1]}}$ (see~\cite[Proposition~2.2]{pa-ve16}). Hence, $\pi_1(P_1\cap P_2)\to\pi_1(P_1)$ is also injective. Therefore, $(B\mb G)^\sK=P_1\cup P_2$ is aspherical by Whitehead's theorem. The isomorphism $\pi_1((B\mb G)^\sK)\cong\mb G^{\sK}$ follows easily from the Seifert--Van Kampen theorem.
\end{proof}

\begin{example}
Let $\sK$ be a flag complex in these examples.

\smallskip

\noindent\textbf{1.} Let $G_i=\Z$, the group of integers, for $i=1,\ldots,m$. The corresponding graph product is known as the \emph{right-angled Artin group} or \emph{partially commutative group} corresponding to~$\sK$ (or to the graph~$\sK^1$). It is given by
\begin{equation}\label{raagdef}
  \raag_\sK = \Z^\sK=F(a_1,\ldots,a_m)\big/ 
  \bigl(a_ia_j=a_ja_i\text{ for
  }\{i,j\}\in\sK\bigr),
\end{equation}
where $F(a_1,\ldots,a_m)$ denotes the free group on $m$ generators. The classifying space of $\raag_\sK$ is the polyhedral product $(B\Z)^\sK=(S^1)^\sK$, a subcomplex of the $m$-torus $T^m=(S^1)^m$. The homotopy fibration~\eqref{hofib} in $\top$ becomes
\[
  (\R,\Z)^\sK\longrightarrow (S^1)^\sK
  \longrightarrow (S^1)^m,
\]
where all spaces are aspherical. Passing to the fundamental groups, we obtain a short exact sequence
\[
  1\longrightarrow\raag_\sK'\longrightarrow\raag_\sK
  \stackrel\Ab\longrightarrow\Z^m\longrightarrow1,
\]
where $\Ab$ is the abelianisation map and $\raag_\sK'$ is the commutator subgroup.

\smallskip

\noindent\textbf{2.} Let $G_i=\Z_2$, a cyclic group of order~$2$, for $i=1,\ldots,m$. The corresponding graph product is known as the \emph{right-angled Coxeter group} corresponding to~$\sK$. It is given by
\begin{equation}\label{racgdef}
  \racg_\sK=\Z_2^\sK=F(a_1,\ldots,a_m)\big/ \bigl(a_i^2=1,\;a_ia_j=a_ja_i
  \text{ for }\{i,j\}\in\sK\bigr).  
\end{equation}
The classifying space of $\racg_\sK$ is the polyhedral product $(B\Z_2)^\sK=(\R P^\infty)^\sK$. The polyhedral product $(E\Z_2,\Z_2)^\sK$ is homotopy equivalent to the \emph{real moment-angle complex} $\rk=(D^1,S^0)^\sK$, where $D^1$ is a 1-disk and $S^0$ is its boundary. The homotopy fibration~\eqref{hofib} in $\top$ becomes
\[
  \rk\longrightarrow (\R P^\infty)^\sK
  \longrightarrow (\R P^\infty)^m,
\]
where all spaces are aspherical. Passing to the fundamental groups, we obtain a short exact sequence
\begin{equation}\label{racges}
  1\longrightarrow\racg_\sK'\longrightarrow\racg_\sK
  \stackrel\Ab\longrightarrow
  \Z_2^m\longrightarrow1.
\end{equation}

\smallskip

\noindent\textbf{3.} We denote the circle by $T$ or $T^1$ when it is viewed as an object in~$\tgp$ and by $S^1$ when it is viewed as an object in~$\top$. Let $G_i=T^1$ for $i=1,\ldots,m$. The corresponding graph product is the colimit of tori $T^I$ in the category of topological groups, known as the \emph{circulation group}~\cite{p-r-v04}:
\[
  \cir_\sK=T^\sK=\colim^\tgp_{I\in\sK}T^I.
\]
Its classifying space is the polyhedral product $(BT^1)^\sK=(\C P^\infty)^\sK$, known as the \emph{Davis--Januszkiewicz space}. The polyhedral product $(ET^1,T^1)^\sK$ is homotopy equivalent to the \emph{moment-angle complex} $\zk=(D^2,S^1)^\sK$, where $D^2$ is a 2-disk. The homotopy fibration~\eqref{hofib} becomes
\[
  \zk\longrightarrow (\C P^\infty)^\sK
  \longrightarrow (\C P^\infty)^m.
\]
Applying the Moore loop space functor $\varOmega\colon\top\to\tmn$ to the fibration above we obtain a commutative diagram in the homotopy category $\mathop{\mathit{Ho}}(\tmn)$
\begin{equation}\label{hofibzk}
\begin{tikzcd}
  &\varOmega\zk \arrow[r] 
  \arrow{d}{\simeq} &
  \varOmega(\C P^\infty)^\sK \arrow[r] \arrow{d}{\simeq} & T^m \arrow[equal]{d} &\\
  1\ar{r} &\cir'_\sK \ar{r} & \cir_\sK \ar{r}{\Ab} & T^m \ar{r} & 1,
\end{tikzcd}
\end{equation}
where $\cir'_\sK$ is the topological commutator subgroup of~$\cir_\sK$. The bottom line is a short exact sequence of strict homomorphisms in~$\tgp$ and is analogous to~\eqref{racges}.

\smallskip

\noindent\textbf{4.} More generally, let $G_i$ be the Moore loop monoid $\varOmega X_i$ of a simply connected space $X_i$, $i=1,\ldots,m$. Then, there is an equivalence in $\mathop{\mathit{Ho}}(\tmn)$
\[
  \varOmega(\mb X^\sK)\simeq(\varOmega\mb X)^\sK= \colim^\tmn_{I\in\sK}(\varOmega\mb X)^I
\]
for flag $\sK$, which follows by applying $\varOmega$ to the homotopy equivalence of Corollary~\ref{bcommut}. This expresses the fact that the Moore loop functor preserves graph products up to homotopy.
\end{example}

Given a commutative ring $R$ with unit, the loop homology $H_*(\varOmega X;R)$ of a simply connected space $X$ is an associative noncommutative algebra with respect to the Pontryagin product. We obtain a functor $H_*\varOmega\colon\top\to\ass$. When $R$ is a field or $H_*(\varOmega X;R)$ is free over $R$, there is a cocommutative coproduct $H_*(\varOmega X;R)\to H_*(\varOmega X;R)\otimes H_*(\varOmega X;R)$ induced by the diagonal map $\Delta\colon\varOmega X\to\varOmega X\times \varOmega X$, which makes $H_*(\varOmega X;R)$ into a cocommutative Hopf algebra~\cite[8.9]{mi-mo65}.

Given a sequence $\mb A=(A_1,\ldots,A_m)$ of associative algebras, the polyhedral product 
\[
  \mb A^\sK=\colim^\ass_{I\in\sK}\bigl(\bigotimes_{i\in I}A_i\bigr)
\]  
is also the graph product, since the category $\ass$ is fat.

\begin{theorem}[{\cite[Corollary~1.2]{dobr}}, see also~{\cite[Theorem~4.2]{cai}}]\label{looppres}
Let $\sK$ be a flag complex, $\mb X=(X_1,\ldots,X_m)$ a sequence of simply connected spaces, and $R$ a field. There are isomorphisms
\[
\begin{aligned}
  H_*\bigl(\varOmega(\mb X^\sK);R\bigr)
  &\cong (H_*(\varOmega\mb X;R))^\sK\\
  &\cong\mathop{\mbox{\Huge$\star$}}_{k=1}^m H_*(\varOmega X_i)\big/\bigl([x_i,x_j]=0\;
  \emph{ for }x_i\in H_*(\varOmega X_i),\,x_j\in H_*(\varOmega X_j),\,\{i,j\}\in\sK\bigr)
\end{aligned}  
\]
of graded $R$-algebras, where $[x_i,x_j]=x_ix_j-(-1)^{|x_i||x_j|}x_jx_i$. 
That is, the loop homology functor $H_*\varOmega\colon\top\to\ass$ with field coefficients preserves graph products.
\end{theorem}

The cases $X_i=S^{2p+1}$ (an odd-dimensional sphere) and $X_i=\C P^\infty$ are of particular importance. We have $H_*(\varOmega S^{2p+1};R)=R[v]$, the polynomial (symmetric) algebra on a single generator $v$ of degree~$2p$. For flag $\sK$, the graph product algebra $H_*(\varOmega(S^{2p+1})^\sK;R)$ is the algebraic counterpart of the right-angled Artin group~\eqref{raagdef}:

\begin{theorem}[\cite{bu-go11}]\label{loopsS}
For any flag complex $\sK$ and $p\ge1$, there are isomorphisms
\begin{equation}\label{lscolim}
\begin{aligned}
  H_*\bigl(\varOmega(S^{2p+1})^\sK;R\bigr)
  &\cong R[v]^\sK=
  \colim^\ass_{I\in\sK}R[v_i\colon i\in I]\\
  &\cong
  T(v_1,\ldots,v_m)/(v_iv_j-v_jv_i=0\emph{ for }\{i,j\}\in\sK)  
\end{aligned}
\end{equation}
of graded $R$-algebras, where $T(v_1,\ldots,v_m)$ is the free associative algebra, $R[v_i\colon i\in I]$ is the polynomial algebra and $|v_i|=2p$. 
\end{theorem}

Similarly, the algebra $H_*(\varOmega(\C P^\infty)^\sK;R)$ is related to right-angled Coxeter groups. To elaborate on this, consider the exterior algebra 
\[
  H_*(\varOmega(\C P^\infty)^m;R)=H (T^m;R)=\Lambda[u_1,\ldots,u_m],
\] 
where $|u_i|=1$. For the homotopy fibration in~\eqref{hofibzk}, we obtain a short exact sequence of loop homology algebras~\cite[Proposition~8.4.1]{bu-pa15}
\begin{equation}\label{lhes}
  R\longrightarrow H_*(\varOmega \zk;R)
  \longrightarrow H_*(\varOmega(\C P^\infty)^\sK;R)
  \longrightarrow 
  \Lambda[u_1,\ldots,u_m]\longrightarrow R,
\end{equation}
that is, $H_*(\varOmega \zk;R)\subset H_*(\varOmega(\C P^\infty)^\sK;R)$ is a normal subalgebra~\cite{mi-mo65} and 
\[
  H_*(\varOmega(\C P^\infty)^\sK;R)/\!/H_*(\varOmega \zk;R)=\Lambda[u_1,\ldots,u_m].
\]
By analogy with~\eqref{racges}, we can view $H_*(\varOmega \zk;R)$ as the \emph{commutator subalgebra} of $H_*(\varOmega(\C P^\infty)^\sK;R)$, or the algebra kernel of the abelianisation map from the noncommutative algebra $H_*(\varOmega(\C P^\infty)^\sK;R)$.


\begin{theorem}[{\cite[Theorem~9.3]{pa-ra08}, \cite[Theorem~8.5.2, Corollary~8.5.3]{bu-pa15}}]\label{padjs}
For any flag complex $\sK$, there are isomorphisms
\begin{equation}\label{lzkcolim}
\begin{aligned}
  H_*(\varOmega(\C P^\infty)^\sK;R)&
  \cong\Lambda[u]^\sK=
  \colim^\ass_{I\in\sK}\Lambda[u_i\colon i\in I]\\
  &\cong T(u_1,\ldots,u_m)/(u_i^2=0,\;
  u_iu_j+u_ju_i=0\emph{ for }\{i,j\}\in\sK)
\end{aligned}
\end{equation}
of graded $R$-algebras, where $\Lambda[u_i\colon i\in I]$ is the exterior algebra and $|u_i|=1$.
\end{theorem}

The colimit decomposition~\eqref{lzkcolim} of the graded algebra $H_*(\varOmega(\C P^\infty)^\sK;R)$ parallels the colimit decomposition $\racg_\sK=\colim^\grp_{I\in\sK}\Z_2^I$ of a right-angled Coxeter group. This analogy can be pursued further by considering the commutator sugroups and the commutator subalgebras.
We denote the group commutator $g^{-1}h^{-1}gh$ by $(g,h)$
and denote the graded algebra commutator $uv-(-1)^{|u||v|}vu$ by $[u,v]$.

\begin{theorem}[{\cite[Theorem~4.5]{pa-ve16}}]\label{basgc}
The commutator
subgroup $\racg'_\sK$ has a finite minimal generator set
consisting of $\sum_{J\subset[m]}\mathop{\mathrm{rank}}\widetilde H_0(\sK_J)$
iterated commutators
\[
  (g_j,g_i),\quad (g_{k_1},(g_j,g_i)),\quad\ldots,\quad
  (g_{k_1},(g_{k_2},\cdots(g_{k_{m-2}},(g_j,g_i))\cdots)),
\]
where $k_1<k_2<\cdots<k_{\ell-2}<j>i$, $k_s\ne i$ for any~$s$, and $i$ is the smallest vertex in a connected component not containing~$j$ of the subcomplex
$\sK_{\{k_1,\ldots,k_{\ell-2},j,i\}}$.
\end{theorem}

\begin{theorem}[{\cite[Theorem~4.3]{g-p-t-w16}}]\label{baslc}
Assume that $\sK$ is flag and $R$ is a field. 
The algebra $H_*(\varOmega\zk;R)$, viewed as the commutator
subalgebra of $H_*(\varOmega(\C P^\infty)^\sK;R)$ via exact sequence~\eqref{lhes},
is multiplicatively generated by $\sum_{J\subset[m]}\mathop{\mathrm{rank}}\widetilde H_0(\sK_J)$ iterated commutators of the form
\[
  [u_j,u_i],\quad [u_{k_1},[u_j,u_i]],\quad\ldots,\quad
  [u_{k_1},[u_{k_2},\cdots[u_{k_{m-2}},[u_j,u_i]]\cdots]]
\]
where $k_1<k_2<\cdots<k_p<j>i$, $k_s\ne i$ for any~$s$, and $i$ is
the smallest vertex in a connected component not containing~$j$ of
the subcomplex $\sK_{\{k_1,\ldots,k_p,j,i\}}$. Furthermore, this
multiplicative generating set is minimal, that is, the commutators
above form a basis in the submodule of indecomposables
in~$H_*( \varOmega\zk;R)$.
\end{theorem}

\section{Central series and associated Lie algebras}

A \emph{Lie algebra} over a commutative unitary ring $R$ is an $R$-module $L$ together with an $R$-linear map $L\times L\to L$, $(x,y)\mapsto[x,y]$, satisfying (1) $[x,x]=0$ and (2) the Jacobi identity $[[x,y],z]+[[y,z],x]+[[z,x],y]=0$ for all $x,y,z\in L$. Condition (1) implies $[x,y]=-[y,x]$. 

There is a functor $\mathcal L\colon\ass\to\lie$ that takes an associative $R$-algebra $A$ to the Lie algebra $\mathcal L A$ on the same $R$-module with the Lie bracket defined by $[x,y]=xy-yx$. The functor $\mathcal L$ has a left adjoint, $\mathcal U\colon\lie\to\ass$, satisfying
\[
  \Hom_{\ass}(\mathcal UL,A)=\Hom_{\lie}(L,\mathcal L A).
\]
The associative algebra $\mathcal U L$ together with the canonical Lie algebra homomorphism $L\to\mathcal L\mathcal UL$ is called the \emph{universal enveloping algebra} of~$L$. Explicitly, $\mathcal UL$ is the quotient of $TL$ (the free associate algebra on~$L$) by the two-sided ideal generated by the elements $xy-yx-[x,y]$, $x,y\in L$. If $L$ is a free $R$-module, then $L\to\mathcal L\mathcal UL$ is injective~\cite[Corollary~III.1]{serr06}.

\smallskip

Let $G$ be a group (if necessary, viewed as an object in $\tgp$ with discrete topology). Given two subgroups $H$ and $W$ of $G$, denote by $(H,W)$ the subgroup generated by all commutators $(h,w)$ with $h\in H$ and $w\in W$. 
In particular, $(G, G)=G'$ is the commutator subgroup.

A (descending) \emph{central series} or \emph{central filtration} on $G$ is a sequence of subgroups $\varGamma(G) = \{ \varGamma_k(G)\colon k\ge 1\}$ such that $\varGamma_1(G)=G$, 
$\varGamma_{k+1}(G)\subset\varGamma_k(G)$ and $(\varGamma_k(G),\varGamma_l(G))\subset\varGamma_{k+\ell}(G)$ for any $k,\ell$. It follows immediately that each $\varGamma_k(G)$ is a normal subgroup and the quotient $\varGamma_k(G)/\varGamma_{k+1}(G)$ is abelian.

The fastest descending central series is the \emph{lower central series} $\gamma(G)$, given by $\gamma_1(G)=G$ and $\gamma_k(G) = (\gamma_{k-1}(G), G)$ for $k\ge2$.

The \emph{associated Lie algebra} of a central series $\varGamma(G)$ is the associated graded abelian group
\[
  \gr_\bullet\varGamma(G)
  =\bigoplus_{k\ge1}\varGamma_k(G)/\varGamma_{k+1}(G)
\]
with the Lie bracket defined by
\[
  [\overline g_k,\overline g_\ell]
  =\overline{(g_k,g_\ell)},
\]
where $\overline g_k$ denotes the image of $g_k\in\varGamma_k(G)$ in the quotient group $\varGamma_k(G)/\varGamma_{k+1}(G)$. 

Given a set $S$, the \emph{free Lie algebra} $\FL(x_s\colon s\in S)$ generated by~$S$ is defined by the universal property: any set map $S\to L$ to a Lie algebra $L$ extends uniquely to a homomorphism  $\FL(x_s\colon s\in S)\to L$ of Lie algebras.
By the classical result of Magnus~\cite[Theorem~5.12]{m-k-s76}, the Lie algebra associated with the lower central series of a free group $F(a_1,\ldots,a_m)$ is the free Lie algebra $\FL(\overline a_1,\ldots,\overline a_m)$ over~$\Z$. 
There is the following generalisation to right-angled Artin groups:

\begin{theorem}[\cite{pa-su06}, see also~\cite{wade15}]\label{lielcfA}
The Lie algebra associated with the lower central series of a right-angled Artin group $\raag_\sK$ is given by
\begin{equation}\label{graag}
  \gr_\bullet\gamma(\raag_\sK)\cong
  \FL(v_1,\ldots,v_m)\big/
  \bigl([v_i,v_j]=0\;
  \emph{for }\{i,j\}\in\sK\bigr).
\end{equation}
Furthermore, each component $\gr_k\gamma(\raag_\sK)$ is a free abelian group.
\end{theorem}

The Lie algebra (over $\Z$) on the right hand side of~\eqref{graag} is known as the \emph{graph Lie algebra}, or \emph{partially commutative Lie algebra}. It is the graph product $\Z\langle v\rangle^\sK$ of trivial Lie algebras $\Z\langle v\rangle=\FL(v)$. Since~\eqref{graag} is $\Z$-free, it is the Lie algebra of primitive elements in 
$\mathcal U(\gr_\bullet\gamma(\raag_\sK))$~\cite[Theorem~III.5.4]{serr06}. Since the universal enveloping functor $\mathcal U\colon\lie\to\ass$ is left adjoint and preserves products, it also preserves graph products. Therefore, $\mathcal U(\gr_\bullet\gamma(\raag_\sK))$ is the algebra~\eqref{lscolim} for $R=\Z$. We obtain

\begin{proposition}\label{ueraag}
For a flag complex $\sK$, there is an isomorphism
\[
  \mathcal U(\gr_\bullet\gamma(\raag_\sK))\cong
  H_*(\varOmega(S^{2p+1})^\sK;\Z)
\]
of associative $\Z$-algebras, where $p\ge1$.
\end{proposition}

To make the isomorphism of Proposition~\ref{ueraag} compatible with the grading of $H_*(\varOmega(S^{2p+1})^\sK;\Z)$ one needs to set $\deg v_i=2p$.

The lower central series of a right-angled Coxeter group $\racg_\sK$ is more subtle~\cite{vere22,ve-ra24}. One has $(\gamma_k(\racg_\sK))^2\subset\gamma_{k+1}(\racg_\sK)$, that is, the square of any element of $\gamma_k(\racg_\sK)$ lies in $\gamma_{k+1}(\racg_\sK)$, so
$\gr_\bullet\gamma(\racg_\sK)$ is a Lie algebra over~$\Z_2$. There is an epimorphism 
\[
  \FL_{\Z_2}(v_1,\ldots,v_m)\big/
  \bigl([v_i,v_j]=0\;
  \text{for }\{i,j\}\in\sK\bigr)\to
  \gr_\bullet\gamma(\racg_\sK)
\]
of Lie algebras over~$\Z_2$, where $\FL_{\Z_2}(v_1,\ldots,v_m)=\FL(v_1,\ldots,v_m)\otimes\Z_2$. This fails to be an isomorphism already when $\sK$ is two disjoint points and $\racg_\sK=\Z_2\star\Z_2$.

For right-angled Coxeter groups, the proper analogue of the isomorphism of Proposition~\ref{ueraag} involves the lower $2$-central series and the associated restricted Lie algebra.

\section{Restricted Lie algebras and $p$-central series}\label{secres}

Let $p$ be a prime and let $R$ be a field of characteristic $p$ or, more generally, a commutative algebra over~$\Z_p$. 
A \emph{restricted Lie algebra} over $R$, or shortly a \emph{$p$-Lie algebra}, is a Lie algebra $L$ over $R$ equipped with a \emph{$p$-power} operation $x \mapsto x^{[p]}$ satisfying the following properties:
\begin{enumerate}
\item $[x, y^{[p]}] = 
[\ldots[[x,\underbrace{y],y], \ldots, y}_p]$  for $x,y\in L$;
\item $(\alpha x)^{[p]} = \alpha^px^{[p]}$ for $x\in L$ and
$\alpha \in R$;
\item 
$\displaystyle
  (x+y)^{[p]} = x^{[p]} + y^{[p]} + \sum_{i=1}^{p-1}s_i(x,y),
$\\
where $s_i(x,y)$ is defined by the expansion
\[
  [\ldots[[x,\underbrace{\lambda x+y],\lambda x+y], \ldots,\lambda x+y}_{p-1}]=\sum_{i=1}^{p-1}is_i(x,y)\lambda^{i-1}
\]
for $x,y\in L$ and $\lambda\in R$.
\end{enumerate}
For example, if $p=2$ then $(x+y)^{[2]}=x^{[2]}+y^{[2]}+[x,y]$, and if $p=3$ then 
$(x+y)^{[3]}=x^{[3]}+y^{[3]}+[[x,y],y]-[[x,y],x]$.
  
An associative algebra $A$ over $R$ is a restricted Lie algebra with the commutator bracket $[x,y]=xy-yx$ and the $p$-power operation $x\mapsto x^p$. This defines a functor $\mathcal L_p\colon\ass\to\lie_p$ to the category of restricted Lie algebras over~$R$. The left adjoint functor $\mathcal U_p\colon\lie_p\to\ass$ defines the \emph{restricted universal enveloping algebra} $\mathcal U_p L$ of $L$ together with the canonical homomorphism $L\to\mathcal L_p\mathcal U_pL$ of restricted Lie algebras. Explicitly, $\mathcal U_pL$ is the quotient of the free associatie algebra $TL$ by the two-sided ideal generated by the elements $x^p-x^{[p]}$ and $xy-yx-[x,y]$, $x,y\in L$. If $L$ is a free $R$-module (e.g., if $R$ is a field of characteristic~$p$), then $L\to\mathcal L_p\mathcal U_pL$ is injective~\cite[Theorem~V.12]{jaco79}.

\begin{lemma}[{\cite[Lemma~2.1]{quil68}}]\label{UnivAndLie}
Let $f\colon L_1 \rightarrow L_2$ be a homomorphism of $R$-free $p$-Lie algebras. Then, $f$ is surjective (injective) if and only if\, $\mathcal U_pf\colon \mathcal U_p L_1 \rightarrow \mathcal U_pL_2$ is surjective (injective).
\end{lemma}

Given a set $S$, the \emph{free $p$-Lie algebra} $\FL_p(x_s\colon s\in S)$ generated by~$S$ is defined by the universal property: any set map $S\to L$ to a $p$-Lie algebra $L$ extends uniquely to a homomorphism  $\FL_p(x_s\colon s\in S)\to L$ of $p$-Lie algebras.

\begin{proposition}[{\cite[\S2.7.1]{bakh85}}]\label{basrl}
If $W$ is an $R$-basis of the free Lie algebra $\FL(x_s\colon s\in S)$, then $\{w^{p^i}\colon w\in W,\,i=0,1,\ldots\}$ is an $R$-basis of $\FL_p(x_s\colon s\in S)$.
\end{proposition}

The restricted universal enveloping algebra of a $p$-Lie algebra is a Hopf algebra~\cite{mi-mo65}.
Let $\mathcal P\colon\hpf\to\lie_p$ denote the primitive elements functor from the category of Hopf algebras over~$R$.

\begin{proposition} $\mathcal U_p\FL_p(x_s\colon s\in S)\cong T(x_s\colon s\in S)$ and $\FL_p(x_s\colon s\in S)\cong\mathcal PT(x_s\colon s\in S)$.
\end{proposition}
\begin{proof}
The first statement is a formal corollary of universal properties. Namely, a set map $S\to A$ to an associative algebra $A$ extends uniquely to a homomorphism $\FL_p(x_s\colon s\in S)\to\mathcal L_p A$  of $p$-Lie algebras. The latter extends uniquely to a homomorphism $\mathcal U_p\FL_p(x_s\colon s\in S)\to A$ of associative algebras. Hence, $\mathcal U_p\FL_p(x_s\colon s\in S)$ has the universal property of a free object in $\ass$, so it is canonically isomorphic to $T(x_s\colon s\in S)$.

For the second statement, we have $\mathcal U\FL(x_s\colon s\in S)=T(x_s\colon s\in S)$. Let $W$ be an $R$-basis of  $\FL(x_s\colon s\in S)$; it follows from~\cite[Theorem~III.5.4, Exercise~2]{serr06} that $\mathcal PT(x_s\colon s\in S)$ has $R$-basis 
$\{w^{p^i}\colon w\in W,\,i=0,1,\ldots\}$. The latter is also an $R$-basis of $\FL_p(x_s\colon s\in S)$ by Proposition~\ref{basrl}, so the result follows.
\end{proof}

More generally, $\mathcal P\mathcal U_pL=L$ for an $R$-free $p$-Lie algebra~$L$~\cite[Theorem~6.11]{mi-mo65}.

We refer to the monographs~\cite{jaco79} and~\cite{bakh85} for more details on $p$-Lie algebras. 

\smallskip

A \emph{$p$-central series} or a \emph{$p$-central filtration}~\cite{laza54,quil68} on a group $G$ is a central series $\varGamma^{[p]}(G) = \{ \varGamma^{[p]}_k(G)\colon k\ge 1\}$ such that if $g\in \varGamma^{[p]}_k(G)$, then $g^p\in \varGamma_{pk}^{[p]}(G)$ for all~$k$.
Each quotient $\varGamma^{[p]}_k(G)/\varGamma^{[p]}_{k+1}(G)$ is an elementary abelian $p$-group (a $\Z_p$-module).

The \emph{lower $p$-central series} is the fastest descending $p$-central series, that is, the $p$-central series $\gamma^{[p]}(G)=\{\gamma^{[p]}_k(G)\colon k\ge 1\}$ with the property that $\gamma^{[p]}_k(G)\subset \varGamma^{[p]}_k(G)$ for any $p$-central series $\varGamma^{[p]}(G)$ and $k\ge 1$. It is given by $\gamma^{[p]}_1(G)=G$ and $\gamma^{[p]}_k(G) = \bigl(\gamma^{[p]}_{k-1}(G), G\bigr)\bigl(\gamma^{[p]}_{\lceil k/p\rceil}\bigr)^p$ for $k\ge2$, or, more explicitly,
\[
  \gamma^{[p]}_k(G) =\prod_{mp^i\ge k}\bigl(\gamma_m(G)\bigr)^{p^i}.
\]

The graded $\Z_p$-module
\[
  \gr_\bullet\varGamma^{[p]}(G)
  =\bigoplus_{k\ge1}
  \varGamma^{[p]}_k(G)/\varGamma^{[p]}_{k+1}(G)
\]
associated with a $p$-central series $\varGamma^{[p]}(G)$ is a restricted Lie algebra over~$\Z_p$ with the bracket and $p$-power defined by
\[
  [\overline g_k,\overline g_\ell]
  =\overline{(g_k,g_\ell)},\quad (\overline g_k)^{[p]}=
  \overline{g_k^p}
\]
for $g_k\in\varGamma^{[p]}_k(G)$, $g_\ell\in\varGamma^{[p]}_\ell(G)$.

\section{Group algebras}
Given a commutative unitary ring $R$ and a group $G$, the \emph{group algebra} $R[G]$ is the associative algebra
whose underlying $R$-module is the the free module with basis $\{g\in G\}$ and multiplication is given on basis elements by the group operation. The group algebra functor $\grp\to\ass$ is left adjoint to the group of units functor $\ass\to\grp$:
\begin{equation}\label{adjga}
  \Hom_{\ass}(R[G],A)=\Hom_{\grp}(G,A^\times),
\end{equation}
where $A^\times$ denotes the group of invertible elements in an associative algebra~$A$. This defines the universal property of the group algebra.

\begin{proposition}\label{gagp}
The group algebra functor $\grp\to\ass$ preserves graph products, that is, $R[\mb G^\sK]\cong(R[\mb G])^\sK$.
\end{proposition}
\begin{proof}
The group algebra functor preserves products and, since it is left adjoint, it preserves colimits.
\end{proof}

\begin{proposition}
Suppose a group $G$ is given by generators and relations:
\[
  G = F(a_i\colon i\in I)/(r_j=1\;\emph{for }j\in J).
\]
Then
\[
  R[G]\cong T(v_i,v_i^{-1}\colon i\in I)/(v_iv_i^{-1}=1 \;\emph{for } i\in I,\;R_j=1\;\emph{for } j\in J),
\]
where $R_j$ is the free monomial in $v_i,v_i^{-1}$ corresponding to the word~$r_j$.
\end{proposition}
\begin{proof}
Let $V$ denote the associative algebra on the right hand side of the isomorphism being proved. We need to check that $V$ satisfies the universal property given by the adjunction~\eqref{adjga}. There is a homomorphism $i\colon G\to V^\times$ defined on the generators by~$i(a_i)=v_i$. We need to show that for any homomorphism $f\colon G\to A^\times$ from $G$ to the group of units of an algebra~$A$ there is a unique homomorphism $F\colon V\to A$ such that $F^\times\circ i=f$. The latter identity implies that $F(v_i)=f(a_i)$ and $F(v_i^{-1})=(f(a_i))^{-1}$, which defines the homomorphism $F$ uniquely.
\end{proof}

\begin{example}\label{Z2RCK}
Applying either of the two propositions above to the right-angled Coxeter group~\eqref{racgdef} we obtain
\[
  \Z_2[\RCK] \cong T_{\Z_2}(v_1,\ldots,v_m)/ 
  (v_i^2=1,\;v_iv_jv_iv_j=1
  \text{ for }\{i,j\}\in\sK).
\]  
Note that the relation $v_iv_jv_iv_j=1$ is equivalent to $v_iv_j+v_jv_i=0$.
\end{example}

The augmentation homomorphism $\varepsilon\colon R[G]\to R$ is given by $\varepsilon(\sum_i k_ig_i) = \sum_i k_i$, where $k_i\in R$ and $g_i\in G$. Let $\overline{R[G]}=\Ker\varepsilon$ be the augmentation ideal, and let 
\[
  \gr_\bullet(R[G]) = \bigoplus_{n\ge 1} (\overline{R[G]})^n /(\overline{R[G]})^{n+1}
\]
be the associated graded ring for the $\overline{R[G]}$-adic filtration.

\begin{theorem}[\cite{quil68}] 
\label{UnivLCS}    
If $R$ is a field of characteristic~$p$, then there is an isomorphism 
\[
  \mathcal U_p\bigl(\gr_\bullet\gamma^{[p]}(G)\bigr)\otimes_\Z R\cong \gr_\bullet(R[G])
\]
of graded $R$-algebras.
\end{theorem}

\section{Lower $p$-central series of graph products of elementary abelian $p$-groups}

Here we consider the right-angled Coxeter group $\racg_\sK=\Z_2^\sK$ and, more generaly, the graph product of groups~$\Z_p$:
\[
  \Z_p^\sK=\colim^\grp_{I\in\sK}\Z_p^I=F(a_1,\ldots,a_m)\big/ \bigl(a_i^p=1,\;a_ia_j=a_ja_i
  \text{ for }\{i,j\}\in\sK\bigr).
\]

The restricted Lie algebra associated with the lower $p$-central series of~$\Z_p^\sK$ is described next:

\begin{theorem}\label{rlazpk}
For any simplicial complex $\sK$ and prime $p$, there is an isomorphism
\[
  \gr_\bullet\gamma^{[p]}\bigl(\Z_p^\sK\bigr)\cong
  \FL_p(u_1,\ldots,u_m)\big/\bigl(u_i^{[p]}=0,\; 
  [u_i,u_j]=0\emph{ for }\{i,j\}\in\sK\bigr)
\]
of $p$-Lie algebras over~$\Z_p$.
\end{theorem}

\begin{proof}
We denote by $L_\sK^{[p]}$ the algebra on the right side of the isomorphism. Let $\Z_p\langle u\rangle=\FL_p(u)/(u^{[p]}=0)$ be the trivial $p$-Lie algebra with one generator. Then $L_\sK^{[p]}$ is the graph product of $\Z_p\langle u\rangle$:
\[
  L_\sK^{[p]}=
  \FL_p(u_1,\ldots,u_m)\big/\bigl(u_i^{[p]}=0,\; 
  [u_i,u_j]=0\text{ for }\{i,j\}\in\sK\bigr)=\bigl(\Z_p\langle u\rangle\bigr)^\sK.
\]

By Lemma~\ref{UnivAndLie}, it is enough to prove that 
\begin{equation}\label{2uea}
  \mathcal U_p\bigl(\gr_\bullet\gamma^{[p]}\bigl(\Z_p^\sK)\bigr)
  \cong\mathcal U_p\bigl(L_\sK^{[p]}\bigr).
\end{equation}
Since the universal enveloping functor $\mathcal U_p\colon\lie_p\to\ass$ is left adjoint and preserves products, it also preserves graph products. Therefore,
\begin{equation}\label{pueapres}
  \mathcal U_p(L_\sK^{[p]})=
  T_{\Z_p}(u_1,\ldots,u_m)\big/\bigl(u_i^p=0,\; 
  u_iu_j=u_ju_i\text{ for }\{i,j\}\in\sK\bigr).
\end{equation}
Now consider the left hand side of~\eqref{2uea}. By Theorem~\ref{UnivLCS}, there is an isomorphism
\[
  \mathcal U_p\bigl(\gr_\bullet\gamma^{[p]}\bigl(\Z_p^\sK)\bigr)
  \cong\gr_\bullet\bigl(\Z_p[\Z_p^\sK]\bigr).
\]
The associated graded algebra on the right hand side is taken with respect to the filtration of the group algebra $\Z_p[\Z_p^\sK]$ by the powers of the augmentation ideal $\overline{\Z_p[\Z_p^\sK]}=\Ker\varepsilon$. By virtue of Proposition~\ref{gagp},
\begin{equation}\label{pgapres}
  \Z_p[\Z_p^\sK]\cong T_{\Z_p}(v_1,\ldots,v_m)/ 
  (v_i^p=1,\;v_iv_j=v_jv_i\text{ for }\{i,j\}\in\sK),
\end{equation}
compare Example~\ref{Z2RCK}. The augmentation $\varepsilon\colon\Z_p[\Z_p^\sK]\to\Z_p$ is given by $v_i\mapsto1$, $i=1,\ldots,m$. Comparing~\eqref{pueapres} with~\eqref{pgapres} we observe that
\[
  \bigl(\mathcal U_p(L_\sK^{[p]}),\varepsilon'\bigr)\to
  \bigl(\Z_p[\Z_p^\sK],\varepsilon\bigr), \quad 
  u_i\mapsto v_i-1
\]
is an isomorphism of augmented algebras, where $\varepsilon'\colon \mathcal U_p(L_\sK^{[p]})\to\Z_p$ is given by $u_i\mapsto0$. Now, $\mathcal U_p(L_\sK^{[p]})$ is the graded quotient of the free associative algebra~\eqref{pueapres} with the grading given by $\deg u_i=1$. The filtration of $\mathcal U_p(L_\sK^{[p]})$ by the powers of the augmentation ideal $\Ker\varepsilon'$ coincides with the filtration defined by the grading. For the latter filtration, the associated graded algebra
$\gr_\bullet\bigl(\mathcal U_p(L_\sK^{[p]})\bigr)$ is $\mathcal U_p(L_\sK^{[p]})$ itself. It follows that
\[
  \mathcal U_p(L_\sK^{[p]})\cong
  \gr_\bullet\bigl(\mathcal U_p(L_\sK^{[p]}),\varepsilon'\bigr)\cong
  \gr_\bullet\bigl(\Z_p[\Z_p^\sK],\varepsilon\bigr)\cong
  \mathcal U_p\bigl(\gr_\bullet\gamma^{[p]}\bigl(\Z_p^\sK)\bigr),
\]
proving~\eqref{2uea} and the theorem.
\end{proof}

Theorem \ref{rlazpk} says that the functor $\gr_\bullet\gamma^{[p]}\colon\grp\to\lie_p$ preserves graph products of cyclic groups $\Z_p$. This result can be extended to a more general class of groups using the following construction.

Let $\sK$ be a simplicial complex on the set~$[m]$, and let $\sK_i$ be a simplicial complex on a finite set set~$V_i$ for $i=1,\ldots,m$. The \emph{substitution complex}~\cite{ab-pa19} is the following simplicial complex on the set $V_1\sqcup\cdots\sqcup V_m$:
\[
  \sK(\sK_1,\ldots,\sK_m)=\bigl\{I_{j_1}\sqcup\cdots\sqcup I_{j_k}\colon
  I_{j_l}\in\sK_{j_l},\;l=1,\ldots,k,\quad\text{and}\quad
  \{j_1,\ldots,j_k\}\in\sK\bigr\}.
\]
When $\sK$ is $m$ disjoint points, we have $\sK(\sK_1,\ldots,\sK_m)=\sK_1\sqcup\cdots\sqcup\sK_m$. When $\sK=\Delta_{[m]}$, the substitution complex is the join: $\sK(\sK_1,\ldots,\sK_m)=\sK_1*\cdots*\sK_m$.

Assume given sequences $\mb X_i=(X_{v,i}\colon v\in V_i)$ of objects in $\cat$ indexed by the elements of $V_i$ for $i=1,\ldots,m$.

\begin{proposition}\label{ppsubst}
The following isomorphism holds for the polyhedral product~\eqref{mcprod}:
\[
  (\mb X_1,\ldots,\mb X_m)^{\sK(\sK_1,\ldots,\sK_m)}\cong
  \bigl(\mb X_1^{\sK_1},\ldots,\mb X_m^{\sK_m}\bigr)^{\sK}.
\]
\end{proposition}
\begin{proof}
We have
\begin{multline*}
  (\mb X_1,\ldots,\mb X_m)^{\sK(\sK_1,\ldots,\sK_m)}=
  \colim^\cat_{I_{j_1}\sqcup\cdots\sqcup I_{j_k}
  \in\sK(\sK_1,\ldots,\sK_m)}
  (\mb X_1,\ldots,\mb X_m)^{I_{j_1}\sqcup\cdots\sqcup I_{j_k}}
  \\=\colim^\cat_{I_{j_1}\sqcup\cdots\sqcup I_{j_k}
  \in\sK(\sK_1,\ldots,\sK_m)}
  \mb X_{j_1}^{I_{j_1}}\times\cdots\times\mb X_{j_k}^{I_{j_k}}
  \\\cong
  \colim^\cat_{\{j_1,\ldots,j_k\}\in\sK}\Bigl(
  \bigl(\colim^\cat_{I_{j_1}\in\sK_{j_1}}\mb X_{j_1}^{I_{j_1}}\bigr)\times\cdots\times
  \bigl(\colim^\cat_{I_{j_k}\in\sK_{j_k}}
  \mb X^{I_{j_k}}_{j_k}\bigr)\Bigr)\\=
  \colim^\cat_{\{j_1,\ldots,j_k\}\in\sK}\bigl(
  \mb X_{j_1}^{\sK_{j_1}}\times\cdots\times
  \mb X_{j_k}^{\sK_{j_k}}\bigr)=
  \bigl(\mb X_1^{\sK_1},\ldots,\mb X_m^{\sK_m}\bigr)^{\sK}.\qedhere
\end{multline*}
\end{proof}

\begin{proposition}
The functor $\gr_\bullet\gamma^{[p]}\colon\grp\to\lie_p$ preserves graph products of elementary abelian $p$-groups.
\end{proposition}

\begin{proof}
Let $\mb G^\sK=(G_1,\ldots,G_m)^\sK$ be the graph product of elementary abelian $p$-groups $G_i=\Z_p^{n_i}$. Then $\gr_\bullet\gamma^{[p]}(G_i)=\Z_p\langle u_1,\ldots,u_{n_i}\rangle$, the trivial $p$-Lie algebra on $n_i$ generators. Note that $G_i$ is the graph product $\Z_p^{\Delta_{[n_i]}}$, and $\Z_p\langle u_1,\ldots,u_{n_i}\rangle$ is the graph product $\Z_p\langle u\rangle^{\Delta_{[n_i]}}$, where $\Delta_{[n_i]}$ is a simplex on $n_i$ vertices. Consider the substitution complex $\sK(\Delta_{[n_1]},\ldots,\Delta_{[n_m]})$. We have
\[
  \mb G^\sK=\bigl(\Z_p^{\Delta_{[n_1]}},\ldots,
  \Z_p^{\Delta_{[n_m]}}\bigr)^\sK
  \cong\Z_p^{\sK(\Delta_{[n_1]},\ldots,\Delta_{[n_m]})}
\]
by Proposition~\ref{ppsubst}. Applying the functor $\gr_\bullet\gamma^{[p]}$ we obtain 
\begin{multline*}
 \gr_\bullet\gamma^{[p]}\bigl((G_1,\ldots,G_m)^\sK\bigr)\cong \gr_\bullet\gamma^{[p]}\bigl(\Z_p^{\sK(\Delta_{[n_1]},\ldots,
 \Delta_{[n_m]})}\bigr)\cong
 \bigl(\Z_p\langle u\rangle)^{\sK(\Delta_{[n_1]},\ldots,
 \Delta_{[n_m]})}\\
 \cong\bigl(\Z_p\langle u\rangle^{\Delta_{[n_1]}},\ldots,
 \Z_p\langle u\rangle^{\Delta_{[n_m]}}\bigr)^\sK=
 \bigl(\gr_\bullet\gamma^{[p]}(G_1),\ldots,
 \gr_\bullet\gamma^{[p]}(G_m)\bigr)^\sK,
\end{multline*}
where the second isomorphism is by Theorem~\ref{rlazpk} and the third by Proposition~\ref{ppsubst}.
\end{proof}

In the case of right-angled Coxeter groups we obtain

\begin{theorem}\label{2lieracg}
For any simplicial complex $\sK$, there is an isomorphism
\[
  \gr_\bullet\gamma^{[2]}(\racg_\sK)\cong
  \FL_2(u_1,\ldots,u_m)\big/\bigl(u_i^{[2]}=0,\; 
  [u_i,u_j]=0\emph{ for }\{i,j\}\in\sK\bigr)
\]
of $2$-Lie algebras over~$\Z_2$.
\end{theorem}

We can also give analogue of Proposition~\ref{ueraag} for right-angled Coxeter groups:

\begin{proposition}\label{ueracg}
For a flag complex $\sK$, there is an isomorphism
\[
  \mathcal U_2\bigl(\gr_\bullet\gamma^{[2]}(\racg_\sK)\bigr)\cong
  H_*\bigl(\varOmega(\C P^\infty)^\sK;\Z_2\bigl)
\]
of associative $\Z_2$-algebras, where the loop homology algebra on the right hand side is given by~\eqref{lzkcolim}.
\end{proposition}

To make the isomorphism of Proposition~\ref{ueracg} compatible with the grading of $H_*(\varOmega(\C P^\infty)^\sK);\Z_2)$ one needs to set $\deg u_i=1$.

\section*{Appendix. Graph products of groups and algebras: a comparison}\label{sectable}

Group-theoretic results and constructions are given in the left column, and parallel homotopy-theoretic results on loop homology and graph products of algebras are given in the right column. Four rows spanning both columns represent results linking graph products of groups with graph products of algebras and loop homology. For both columns, $\mb X^\sK=\colim^\top_{I\in\sK}\mb X^I$ denotes the polyhedral product of a sequence of pointed topological spaces $\mb X=(X_1,\ldots,X_m)$.

\begin{longtable}{p{0.48\textwidth}|p{0.48\textwidth}}

\hline\\

$\mb G=(G_1,\ldots,G_m)$ a sequence of groups & 
$\mb A=(A_1,\ldots,A_m)$ a sequence of associative algebras with unit over $R$\\[3pt]

\hline\\

Graph product
$\mb G^\sK=\colim^\grp_{I\in\sK}\mb G^I$ &
Graph product
$\mb A^\sK=\colim^\ass_{I\in\sK}\mb A^{\otimes I}$\\[5pt]

\hline\\

The classifying space functor $B\colon\tgp\to\top$ and Moore loop functor $\varOmega\colon\top\to\tmn$ preserve graph products up to homotopy: &

The loop homology functor $H_*\varOmega\colon\top\to\ass$ with field coefficients preserves graph products:\\[1.1\baselineskip]

\qquad$B(\mb G^{\sK})\simeq(B\mb G)^\sK\text{ and } 
\varOmega(\mb X^\sK)\simeq(\varOmega\mb X)^\sK$ &

\qquad$H_*(\varOmega(\mb X^\sK);R)\cong
(H_*(\varOmega\mb X;R))^\sK$\\[2pt]

for flag $\sK$ (Corollary~\ref{bcommut}). &
for flag $\sK$ (Theorem~\ref{looppres}).\\[3pt]

\hline\\

Right-angled Artin group $\raag_\sK=\Z^\sK$ &
Graph product algebra $R[v]^\sK$\\

$=F(a_1,\ldots,a_m)/(a_ia_j=a_ja_i\text{ for
}\{i,j\}\in\sK)$ &
$=T(v_1,\ldots,v_m)/(v_iv_j=v_jv_i\text{ for }\{i,j\}\in\sK)$\\[3pt]

\hline\\

$B(\raag_\sK)\simeq(S^1)^\sK$ (Proposition~\ref{basph}) &
$(S^{2p+1})^\sK$, $p\ge1$\\[3pt]

\hline\\

$\pi_1((S^1)^\sK)\cong\raag_\sK$ & 
$H_*(\varOmega(S^{2p+1})^\sK;R)\cong R[v]^\sK$
(Theorem~\ref{loopsS})\\[3pt]
\hline

\multicolumn{2}{c}{}\\

\multicolumn{2}{c}{$\gr_\bullet\gamma(\raag_\sK)\cong
\FL(v_1,\ldots,v_m)\big/\bigl([v_i,v_j]=0\;
\text{for }\{i,j\}\in\sK\bigr)$ 
(Theorem~\ref{lielcfA})}\\[5pt]

\hline

\multicolumn{2}{c}{}\\

\multicolumn{2}{c}{$\mathcal U\bigl(\gr_\bullet\gamma(\raag_\sK)\bigr)\cong
H_*\bigl(\varOmega(S^{2p+1})^\sK;\Z\bigr)$ for flag $\sK$ (Proposition~\ref{ueraag})}\\[5pt]

\hline\\

Right-angled Coxeter group $\racg_\sK=\Z_2^\sK=$ &
Graph product algebra $\Lambda[u]^\sK=$\\

$F(a_1,\ldots,a_m)/(a_i^2=1, a_ia_j=a_ja_i,
\{i,j\}\in\sK)$ &
$T(u_1,\ldots,u_m)/(u_i^2=0,u_iu_j+u_ju_i=0,
\{i,j\}\in\sK)$\\[3pt]

\hline\\

$B(\racg_\sK)\simeq(\R P^\infty)^\sK$ (Proposition~\ref{basph}) &
$(\C P^\infty)^\sK$\quad the Davis--Januszkiewicz space\\[3pt]

\hline\\

$\pi_1((\R P^\infty)^\sK)\cong\racg_\sK$ & 
$H_*(\varOmega(\C P^\infty)^\sK;R)\cong \Lambda[u]^\sK$
(Theorem~\ref{padjs})\\[3pt]

\hline\\

$\rk=(D^1,S^0)^\sK$\; the real moment-angle complex & $\zk=(D^2,S^1)^\sK$\quad the moment-angle complex\\[3pt]

\hline\\

Homotopy fibration of aspherical spaces &
Homotopy fibration\\[2pt]
\qquad$\mathcal R_\sK\to(\R P^\infty)^\sK\to (\R P^\infty)^m$
&
\qquad$\mathcal Z_\sK\to(\C P^\infty)^\sK\to (\C P^\infty)^m$
\\[3pt]

\hline\\
Short exact sequence of fundamental groups & 
Short exact sequence of loop homology algebras\\[2pt]
\qquad$1\to\racg_\sK'\to\racg_\sK
  \stackrel\Ab\longrightarrow\Z_2^m\to1$
& 
$R\to H_*(\varOmega \zk)\to 
H_*(\varOmega(\C P^\infty)^\sK)\to 
  \Lambda[m]\to R$\\[3pt]

\hline\\

$\pi_1(\rk)$ is the commutator subgroup of $\pi_1((\R P^\infty)^\sK)=\racg_\sK$ & $H_*(\varOmega\zk)$ is the commutator subalgebra of $H_*(\varOmega(\C P^\infty)^\sK)\cong \Lambda[u]^\sK$\\[3pt]  

\hline\\

A minimal basis for 
$\racg'_\sK=\pi_1(\rk)$ consisting of nested group commutators (Theorem~\ref{basgc}) &
A minimal basis for $H_*(\varOmega\zk)$ consisting of nested Lie commutators (Theorem~\ref{baslc})\\[3pt]

\hline\\

$\racg'_\sK=\pi_1(\rk)$ is a free group if and only if $\sK^1$ is a chordal graph~\cite[Theorem~4.3]{pa-ve16} &
For flag~$\sK$,
$H_*(\varOmega\zk)$ is a free algebra if and only if $\sK^1$ is a chordal graph~\cite[Theorem~4.6]{g-p-t-w16}\\[3pt]

\hline

\multicolumn{2}{c}{}\\

\multicolumn{2}{c}{$\gr_\bullet\gamma^{[2]}(\racg_\sK)\cong
\FL_{\Z_2}(u_1,\ldots,u_m)\big/\bigl(u_i^{[2]}=0,\;[u_i,u_j]=0\;\text{for }\{i,j\}\in\sK\bigr)$ 
(Theorem~\ref{2lieracg})}\\[5pt]

\hline

\multicolumn{2}{c}{}\\

\multicolumn{2}{c}{$\mathcal U_2\bigl(\gr_\bullet\gamma^{[2]}(\racg_\sK)\bigr)\cong
H_*\bigl(\varOmega(\C P^\infty)^\sK;\Z_2\bigr)$ for flag $\sK$ (Proposition~\ref{ueracg})}

\end{longtable}

\end{document}